\documentclass[12pt, reqno]{amsart}

\usepackage{etoolbox, xcolor, enumitem}

\makeatletter
\patchcmd\maketitle
{\uppercasenonmath\shorttitle}
{}
{}{}
\patchcmd\maketitle
{\@nx\MakeUppercase{\the\toks@}}
{\the\toks@}
{}
{}{}

\patchcmd\@settitle{\uppercasenonmath\@title}{\Large}{}{}

\patchcmd\@setauthors
{\MakeUppercase{\authors}}
{\authors}
{}{}
\makeatother

\usepackage{amsmath, amssymb, amsthm}
\usepackage{color}
\usepackage{url}
\usepackage{tikz-cd}
\usepackage{cite}
\usepackage{breakurl}

\usepackage[utf8]{inputenc}
\usepackage[T1]{fontenc}

\usepackage[a4paper, top=1in, bottom=1.5in, left=1.2in, right=1in]{geometry}

\newtheorem{theorem}{Theorem}[section]

\newtheorem{corollary}[theorem]{Corollary}
\newtheorem{proposition}[theorem]{Proposition}

\newtheorem{lemma}[theorem]{Lemma}

\theoremstyle{remark}
\newtheorem*{remark}{\textup{\textbf{Remark}}}

\numberwithin{equation}{section}

\usepackage[colorlinks=true]{hyperref}
\hypersetup{urlcolor=blue, citecolor=red , linkcolor= blue}

\usepackage[capitalise,noabbrev,nameinlink]{cleveref}

\usepackage{etoolbox}

\DeclareMathAlphabet\EuScript{U}{eus}{m}{n}
\DeclareMathAlphabet\EuScriptBold{U}{eus}{b}{n}
\DeclareMathAlphabet\Eurm{U}{eur}{m}{n}
\DeclareMathAlphabet\Eurb{U}{eur}{b}{n}
\newcommand{\mathcalb}{\EuScriptBold}

\def\R{\mathbb{R}}
\def\C{\mathbb{C}}

\def\H{\mathcal{H}}
\def\BH{\mathcalb{B}(\mathcal{H})}

\def\N{\mathbb{N}}
\def\R{\mathbb{R}}
\def\C{\mathbb{C}}

\def\D{\mathbb{D}}
\def\wq{\omega_q}

\newcommand{\doi}[1]{\url{https://doi.org/#1}}

\newcommand\norm[1]{\left\lVert#1\right\rVert}
\newcommand\abs[1]{\lvert#1\rvert}
\newcommand\scal[2]{\left\langle #1,#2\right\rangle}

\begin{document}

		

\date{\today}
		
\title[$q$-numerical radius of rank-one operators and the generalized Buzano inequality]
{$q$-numerical radius of rank-one operators\\and the generalized Buzano inequality}

\author[D.\ Den\v ci\' c, H.\ Stankovi\' c,  M.\ Krsti\' c and I.\ Damnjanovi\' c]{{Du\v san Den\v ci\' c}$^{1}$, {Hranislav Stankovi\' c}$^{2}$, {Mihailo Krsti\' c}$^{3}$ and {Ivan Damnjanovi\' c}$^{2,4}$}
		
\address{$^{[1]}$ Faculty of Agriculture, University of Belgrade, Nemanjina 6, Belgrade, 11080, Serbia.}
\email{\url{dusan.dencic@agrif.bg.ac.rs}}

\address{$^{[2]}$ Faculty of Electronic Engineering, University of Ni\v{s}, Aleksandra Medvedeva 4, Ni\v{s}, 18104, Serbia.}
\email{\url{hranislav.stankovic@elfak.ni.ac.rs}, \url{ivan.damnjanovic@elfak.ni.ac.rs}}

\address{$^{[3]}$ Faculty of Mathematics, University of Belgrade, Studentski trg 16, Belgrade, 11000, Serbia.}
\email{\url{mihailo.krstic@matf.bg.ac.rs}}

\address{$^{[4]}$ Diffine LLC, 3681 Villa Terrace, San Diego, CA, 92104, USA.}

\begin{abstract}
Here, we study the \(q\)-numerical radius of rank-one operators on a Hilbert space \(\mathcal{H}\). More precisely, for \(q \in [0,1]\) and \(a, b \in \mathcal{H}\), we establish the formula  
\[
	\wq(a \otimes b) = \frac{1}{2}\left(\|a\|\|b\| + q|\langle a, b \rangle| + \sqrt{1-q^2}\sqrt{\|a\|^2\|b\|^2 - |\langle a, b \rangle|^2}\right),
\]
which represents a generalization of the well-known formula for the numerical radius of a rank-one operator in a Hilbert space, obtained by setting \(q = 1\). As a corollary, we also derive a generalization of the classical Buzano inequality.
\end{abstract}

\maketitle

\section{Introduction}

Throughout the paper, we consider all Hilbert spaces to be complex. We use $\BH$ to denote the Banach algebra of all bounded linear operators on the Hilbert space $\H$. We also let $\D$ denote the open unit disk in $\mathbb{C}$, that is, 
\[
    \D = \{z \in \mathbb{C} \colon |z| < 1\}.
\]

For a given Hilbert space $\H$, the numerical range of $A \in \BH$ is defined as
\[
\mathcal{W}(A) = \{\langle Ax, x \rangle \colon x \in \H,\, \|x\| = 1\},
\]
while the numerical radius is
\[
	\omega(A) = \sup_{w \in \mathcal{W}(A)} |w|.
\]
Similarly, for any $q \in \overline{\D}$, the $q$-numerical range is defined as
\[
	\mathcal{W}_q(A) = \{\langle Ax, y \rangle \colon x, y \in \H,\, \|x\| = \|y\| = 1,\, \langle x, y \rangle = q\},
\]
while the $q$-numerical radius is
\[
	\omega_q(A) = \sup_{w \in \mathcal{W}_q(A)} |w|.
\]
Note that the $q$-numerical radius generalizes the numerical radius, as $$\omega_q(A) = \omega(A)$$when $|q| = 1$. This follows from the fact that equality must hold in the Cauchy--Schwarz inequality
\[
|q| = |\langle x,y\rangle| \leqslant \|x\|\|y\| = 1,
\]
provided $|q| = 1$. In that case, we have $y = \lambda x$ for some $\lambda \in \C$ with $|\lambda|=1$, hence $|\langle Ax,y\rangle| = |\langle Ax,x\rangle|$.

Clearly, if $\dim(\H)=1$, then $\mathcal{W}_q(A)$ is nonempty if and only if $|q| = 1$, whereas if $\dim(\H)\geqslant 2$, it is easy to see that $\mathcal{W}_q(A)$ is always nonempty. Therefore, in what follows, we restrict our attention to Hilbert spaces of dimension at least two. The following lemma provides elementary properties of the $q$-numerical radius and it is easily derived from \cite[Proposition 3.1]{GauWu}; see also \cite[Lemma~2.2]{q1}.

\begin{lemma}\label{elementary_lemma}
For $A,B\in\BH$ and $\lambda\in\mathbb{C}, q\in\overline{\D}$, we have the following:
\begin{enumerate}[label=\textbf{(\roman*)}]
    \item $\omega_q(\lambda A)=\abs{\lambda} \, \omega_q(A)$;
    \item $\omega_q(A+B)\leqslant \omega_q(A)+\omega_q(B)$;
    \item $\omega_q(U^*AU)=\omega_q(A)$, where $U\in\BH$ is a unitary operator;
    \item $\omega_{\lambda q}(A)=\omega_q(A)$ for all $\lambda\in\C$ with $|\lambda|=1$.
\end{enumerate}
\end{lemma}
By Lemma \ref{elementary_lemma}(iv), it is enough to restrict ourselves to the case $q \in [0, 1]$. The next lemma will also be crucial in the sequel. Its proof can be found in \cite[Lemma~3.1]{q1}.

\begin{lemma}
For any $A\in\BH$ and $q\in\overline{\D}$, we have 
\[
    \mathcal{W}_q(A) = \left\{ q\langle Ay,y\rangle + \sqrt{1-|q|^2} \, \langle At,y\rangle \,\colon \|y\| = \|t\| = 1,\, \langle t,y\rangle = 0 \right\}.
\]
Also, 
\begin{equation}\label{q_num_radius_redef}
    \omega_q(A) = \sup\left\{ |q| \, |\langle Ay,y\rangle| + \sqrt{1-|q|^2} \, |\langle At,y\rangle| \,\colon \|y\| = \|t\| = 1,\, \langle t,y\rangle = 0 \right\}.
\end{equation}
\end{lemma}

For other properties and more details on the $q$-numerical range and $q$-numerical radius, we refer the reader to \cite{Chien2, Chien3, ChienNakazato, Li1, Chien1, Goncalves, Huang, Li2, Li3, Psarrakos, Rajic, MA, Janfada}. Finally, if $a, b \in \H$, then we denote by $a \otimes b$ the $\BH$ operator of rank at most one defined by
\[
    (a\otimes b)x = \langle x,a \rangle\, b.
\]
For an operator of the form $a \otimes b$, we have $(a\otimes b)^* = b \otimes a$, $\|a\otimes b\| = \|a\|\|b\|$, and $\wq(a\otimes b) = \wq(b\otimes a)$ for all $q \in [0,1]$. Based on \cite[Theorem 1]{MF}, we have
\begin{equation}\label{w1}
	\omega(a\otimes b) = \frac{\|a\| \, \|b\|+|\langle a,b\rangle|}{2}.
\end{equation}
 From \eqref{w1}, it is immediate that the inequality
\begin{equation}\label{eq:buzano_ineq}
	\abs{\scal{a}{x}\scal{x}{b}}\leqslant \frac{\|a\| \, \|b\|+|\langle a,b\rangle|}{2}\norm{x}^2
\end{equation}
holds for all $a, b, x \in\mathcal{H}$. Inequality \eqref{eq:buzano_ineq} is known as the Buzano inequality (see \cite{Buzano}) and it is a generalization of the Cauchy--Schwarz inequality. The original proof is somewhat difficult as it requires some facts about the orthogonal decomposition of a complete inner product space. A simpler proof can be found in \cite{MF}. For some refinements of the Buzano inequality in recent years, see \cite{BotConde, Dragomir, KitZam}.

The main purpose of this paper is to derive the analogous formulae to \eqref{w1} and \eqref{eq:buzano_ineq}, for $\wq(a\otimes b)$ for all $q\in[0,1]$. Finally, before we start with the main results, we mention that a part of this manuscript has already appeared inside the doctoral dissertation of one of the authors in \cite{StDis}.

\section{Main results}

We start with the following theorem, which is the main result of the paper and also presents a basis for further developments.
\begin{theorem}\label{osnovna}
	Let $a,b\in\H$ and $q\in[0,1]$. Then we have 
    \begin{equation}\label{glavnaformula}\wq(a\otimes b)=\frac{\|a\| \, \|b\|+q|\langle a,b\rangle|}{2}+\frac{\sqrt{1-q^2}}{2} \, \sqrt{\|a\|^2\|b\|^2-|\langle a,b\rangle|^2}.
	\end{equation}
\end{theorem}

\begin{proof}
	If $a = 0$ or $b = 0$, the theorem holds trivially. 
    
    Now, let $a,b\in\H\setminus\{0\}$ and assume that $\|a\|=\|b\|=1$. From the equality 
	\[
	   \langle (a\otimes b)x, \, y \rangle=\langle x,a\rangle\,\langle b,y\rangle
	\]
	for any $x,y\in\H$ and from \eqref{q_num_radius_redef}, we have
	{\small
    \begin{align}
		\omega_q(a \otimes b)
		&=\sup\Big\{q\,\big|\langle (a\otimes b)y,y\rangle\big|+\sqrt{1-q^2}\,\big|\langle (a\otimes b)t,y\rangle\big| \, \colon \langle t,y\rangle=0,\;\|t\|=\|y\|=1\Big\}\notag\\
		\label{reprezentacija} &= \sup\Big\{q\,\big|\langle a,y\rangle\,\langle y,b\rangle\big|+\sqrt{1-q^2}\,\big|\langle t,a\rangle\,\langle y,b\rangle\big| \, \colon \langle t,y\rangle=0,\;\|t\|=\|y\|=1\Big\}.
	\end{align}}Since $\H=\mathrm{span}\{a\}\oplus\{a\}^{\perp}$, each vector $t,y,b\in\H$ can be represented as
	\begin{align}
		\label{t} t &= \langle t,a\rangle\,a+t_1,\\[1mm]
		\label{y} y &= \langle y,a\rangle\,a+y_1,\\[1mm]
		\label{b} b &= \langle b,a\rangle\,a+b_1,
	\end{align}
	where $t_1,y_1,b_1\in\{a\}^{\perp}$ are uniquely determined. Now, from \eqref{t}--\eqref{b} and $\|a\|=1$, it follows that
	\begin{align}
		\label{norm1} \|t_1\| &= \sqrt{1-|\langle a,t\rangle|^2},\\[1mm]
		\label{norm2} \|y_1\| &= \sqrt{1-|\langle a,y\rangle|^2},\\[1mm]
		\label{norm3} \|b_1\| &= \sqrt{1-|\langle a,b\rangle|^2}.
	\end{align}
	By substituting \eqref{t} and \eqref{y} into $0=\langle t,y\rangle$ and using $\langle a,y_1\rangle=0$ and $\langle a,t_1\rangle=0$, we obtain
	\begin{equation}\label{aux_1}
	0=\langle a,t\rangle\,\langle a,y\rangle+\langle t_1,y_1\rangle.
	\end{equation}
	From \eqref{norm1}, \eqref{norm2}, \eqref{aux_1} and the Cauchy--Schwarz inequality, we have
	\begin{equation*}
		\begin{split}
			|\langle a,t\rangle\,\langle y,a\rangle|
			&=|\langle t_1,y_1\rangle|\\[1mm]
			&\leqslant \|t_1\|\,\|y_1\|\\[1mm]
			&=\sqrt{1-|\langle a,t\rangle|^2}\,\sqrt{1-|\langle a,y\rangle|^2}\\[1mm]
			&=\sqrt{1-|\langle a,t\rangle|^2-|\langle a,y\rangle|^2+|\langle a,t\rangle|^2\,|\langle a,y\rangle|^2}.
		\end{split}
	\end{equation*}
	By squaring the last inequality, we get
	\begin{equation}\label{aux_2}
	   |\langle a,t\rangle|\leqslant \sqrt{1-|\langle a,y\rangle|^2}.
	\end{equation}
	From \eqref{reprezentacija} and \eqref{aux_2}, it follows that
	\[
	   \omega_q(a\otimes b)\leqslant \sup\Big\{q \, |\langle a,y\rangle\,\langle y,b\rangle|+\sqrt{1-q^2}\,\sqrt{1-|\langle a,y\rangle|^2} \, |\langle y,b\rangle| \, \colon \|y\|=1\Big\}.
	\]
	
	Similarly, we obtain the equality
	\begin{equation}\label{aux_3}
	   \langle y,b\rangle=\langle y,a\rangle\,\langle a,b\rangle+\langle y_1,b_1\rangle .
	\end{equation}
	From \eqref{norm2}, \eqref{norm3}, \eqref{aux_3} and the Cauchy--Schwarz inequality, it follows that
	\begin{equation}\label{trogaoiks}
		\begin{split}
			|\langle y,b\rangle|
			&=|\langle y,a\rangle\,\langle a,b\rangle+\langle y_1,b_1\rangle|\\[1mm]
			&\leqslant |\langle y,a\rangle\,\langle a,b\rangle|+|\langle y_1,b_1\rangle|\\[1mm]
			&\leqslant |\langle y,a\rangle\,\langle a,b\rangle|+\|y_1\|\,\|b_1\|\\[1mm]
			&= |\langle y,a\rangle\,\langle a,b\rangle|+ \sqrt{(1-|\langle a,y\rangle|^2)(1-|\langle a,b\rangle|^2)}.
		\end{split}
	\end{equation}
	From this, we obtain
	\begin{equation}\label{nej}
		\omega_q(a\otimes b)\leqslant \sup_{\|y\|=1}I(y),
	\end{equation}
	where we define, for each $y\in\H$,
	{\small\[
	   I(y)=\Bigl(q \, |\langle a,y\rangle|+\sqrt{1-q^2} \, \sqrt{1-|\langle a,y\rangle|^2}\Bigr) \Bigl(|\langle y,a\rangle\,\langle a,b\rangle|+\sqrt{(1-|\langle a,y\rangle|^2)(1-|\langle a,b\rangle|^2)}\Bigr).
	\]}Let $y\in\H$ be an arbitrary unit vector, and let $\alpha,\beta,\gamma\in\left[0,\frac{\pi}{2}\right]$ be such that
	\[
	   q=\cos\alpha,\quad |\langle y,a\rangle|=\cos\beta \quad \text{and}\quad |\langle a,b\rangle|=\cos\gamma.
	\]
	Then, using this notation, we have
	\begin{equation}\label{adicione}
		\begin{split}
			I(y)
			&=(\cos\alpha\cos\beta+\sin\alpha\sin\beta)(\cos\beta\cos\gamma+\sin\beta\sin\gamma)\\[1mm]
			&=\cos(\alpha-\beta)\cos(\beta-\gamma)\\[1mm]
			&=\frac{1}{2}\cos(\alpha-\gamma)+\frac{1}{2}\cos(\alpha+\gamma-2\beta)\\[1mm]
			&\leqslant \frac{1}{2}+\frac{1}{2}\cos(\alpha-\gamma)\\[1mm]
			&=\frac{1}{2}+\frac{1}{2}\Bigl(\cos\alpha\cos\gamma+\sin\alpha\sin\gamma\Bigr)\\[1mm]
			&=\frac{1}{2}+\frac{1}{2}\Bigl(q \, |\langle a,b\rangle|+\sqrt{(1-q^2)(1-|\langle a,b\rangle|^2)}\Bigr).
		\end{split}
	\end{equation}
	Therefore, from \eqref{nej} and \eqref{adicione}, we obtain the estimate
	\begin{align}\label{nejednakost}
		\omega_q(a\otimes b)\leqslant \frac{1}{2}+\frac{1}{2}\Bigl(q \, |\langle a,b\rangle|+\sqrt{(1-q^2)(1-|\langle a,b\rangle|^2)}\Bigr).
	\end{align}

	We now prove that equality actually holds in \eqref{nej}. First, consider the case when $a$ and $b$ are linearly dependent. Then, it is easy to see that  $\omega_q(a\otimes b)=\omega_q(a\otimes a)$. Since $a\otimes a$ is a projection operator, \cite[Theorem 5.7]{q1} immediately implies that the desired equality holds. 
    
    Now assume that $a$ and $b$ are linearly independent. Then, $b_1 \neq 0$, so we can define the vectors
	\[
	   y=\cos\frac{\alpha+\gamma}{2}\,a + \sin\frac{\alpha + \gamma}{2}\,e^{i \, \arg \langle a, b \rangle}\,\frac{b_1}{\|b_1\|}
	\]
    and
    \[
        t = \sin\frac{\alpha+\gamma}{2}\, a - \cos\frac{\alpha + \gamma}{2}\,e^{i \, \arg \langle a, b \rangle}\,\frac{b_1}{\|b_1\|} ,
    \]
	where we take the argument of the complex number zero to be zero. For this choice of vector \(y\in\H\), we have
	\[
	   y_1=\sin\frac{\alpha+\gamma}{2}\,e^{i \, \arg \langle a, b \rangle}\,\frac{b_1}{\|b_1\|} .
	\]
	Since $\sin \frac{\alpha + \gamma}{2} > 0$, we also obtain
	\begin{align}
        \label{p1} \mathrm{arg}\langle y_1,b_1\rangle &= \mathrm{arg}\langle a,b\rangle,\\
        \label{p2} y_1 &\in \mathrm{span}\{b_1\},\\
        \label{beta} \beta &=\frac{\alpha+\gamma}{2}.
	\end{align}
	To attain equality in \eqref{nej}, it suffices that equality holds in \eqref{aux_2}, \eqref{trogaoiks} and \eqref{adicione}. We trivially observe that equality holds in \eqref{aux_2}. By \eqref{p1} and \eqref{p2}, equality holds in \eqref{trogaoiks}, while due to \eqref{beta}, equality also holds in \eqref{adicione}. Thus, equality holds in \eqref{nejednakost}. Therefore, we have
	\[
		\omega_q(a\otimes b)=\frac{1}{2}+\frac{1}{2}\Bigl(q\,|\langle a,b\rangle|+\sqrt{\left(1-q^2\right)\left(1-|\langle a,b\rangle|^2\right)}\Bigr) ,
	\]
    provided $\|a\|=\|b\|=1$.
    
    In the general case, from the homogeneity of \(\omega_q\), we obtain
    {\small
	\begin{align*}\pushQED{\qed}
		\omega_q(a\otimes b) &=\|a\| \, \|b\| \, \omega_q\Bigl(\Bigl(\frac{a}{\|a\|}\Bigr)\otimes\Bigl(\frac{b}{\|b\|}\Bigr)\Bigr)\\[1mm]
		&=\|a\| \, \|b\| \left[\frac{1}{2}+\frac{1}{2}\left(q \, \Bigl|\Bigl\langle\frac{a}{\|a\|},\frac{b}{\|b\|}\Bigr\rangle\Bigr|+\sqrt{\Bigl(1-q^2\Bigr)\Bigl(1-\Bigl|\Bigl\langle\frac{a}{\|a\|},\frac{b}{\|b\|}\Bigr\rangle\Bigr|^2\Bigr)}\,\right)\right]\\[1mm]
		&=\frac{\|a\| \, \|b\|+q\,|\langle a,b\rangle|}{2}+\frac{\sqrt{1-q^2}}{2} \, \sqrt{\|a\|^2 \, \|b\|^2-|\langle a,b\rangle|^2}. \qedhere
	\end{align*}}
\end{proof}

When we set $q=1$ in Theorem \ref{osnovna}, we obtain \eqref{w1}. Moreover, Theorem \ref{osnovna} generalizes the Buzano inequality from \cite{Buzano} as follows.

\begin{corollary}\label{buz_cor}
For any vectors $a, b, x, y \in \H$ and $q \in [0,1]$ such that $\norm{x}=\norm{y}=1$ and $\scal{x}{y}=q$, we have
\begin{equation}\label{genbuz}
		|\scal{a}{x}\scal{b}{y}| \leqslant \frac{\norm{a}\norm{b}+q \, |\langle a,b\rangle|}{2}+\frac{\sqrt{1-q^2}}{2} \, \sqrt{\norm{a}^2\norm{b}^2-|\langle a,b\rangle|^2} .
\end{equation}
\end{corollary}
\begin{remark}
For $q = 1$, we get $y= \lambda x$ for some $\lambda\in\C$ with $|\lambda|=1$, yielding the inequality
\[
    |\scal{a}{x}\scal{x}{b}|\leqslant\frac{\norm{a}\norm{b}+|\langle a,b\rangle|}{2},
\]
which is equivalent to \eqref{eq:buzano_ineq}.
Thus, \eqref{genbuz} represents a generalization of the classical Buzano inequality.
\end{remark}

For any $q \in [0, 1]$ and $T\in\BH$, we trivially observe that $\omega_q(T)\leqslant \norm{T}$. Therefore, provided $T \neq O$, there is a unique $\lambda_q(T) \in [0,1]$ such that $\omega_q(T) = \lambda_q(T) \norm{T}$. The following corollary gives an explicit formula for such a $\lambda_q(T)$.

\begin{corollary}\label{lambda_cor}
Let $a,b\in\H\setminus\{0\}$ and $q\in[0,1]$. Then
\begin{equation}\label{lambda_q}
	\lambda_q(a\otimes b) = \begin{cases}
	    \dfrac{1+\cos\left(\arccos q-\arctan\sqrt{\left(\dfrac{\norm{a}\norm{b}}{|\scal{a}{b}|}\right)^2-1}\right)}{2}, & \scal{a}{b}\neq 0,\\
        \dfrac{1+\sqrt{1-q^2}}{2}, & \scal{a}{b}=0 .
	\end{cases}
\end{equation}
\end{corollary}
\begin{proof}
Observe that \eqref{lambda_q} directly follows from \eqref{glavnaformula} for the case $\scal{a}{b}=0$.

Hence, we assume that $\scal{a}{b}\neq 0$. Let $q=\cos\theta$ for some $\theta\in[0, \frac{\pi}{2}]$. Then, \eqref{glavnaformula} can be written as
	\begin{equation*}
		\omega_q(a\otimes b)=\dfrac{1}{2}\left(\norm{a}\norm{b}+\abs{\scal{a}{b}} \cos\theta + \sqrt{\norm{a}^2\norm{b}^2-\abs{\scal{a}{b}}^2} \, \sin\theta \right) .
	\end{equation*}
	Using the identity
	\begin{equation*}
		a\cos x+b\sin x=\sqrt{a^2+b^2} \, \cos\left(x-\arctan\left( \frac{b}{a} \right)\right)
	\end{equation*}
    that holds for any $a > 0$ and $b \geqslant 0$, we get
	\begin{align*}
		\omega_q(a\otimes b)&=\dfrac{1}{2}\left(\norm{a}\norm{b}+\norm{a}\norm{b}\cos\left(\theta-\arctan\frac{\sqrt{\norm{a}^2\norm{b}^2-\abs{\scal{a}{b}}^2}}{|\scal{a}{b}|}\right)\right),
	\end{align*}
	which directly implies \eqref{lambda_q}.
\end{proof}

\section{Some related results}

In this section, we give some additional results related to Theorem \ref{osnovna} and Corollaries \ref{buz_cor} and \ref{lambda_cor}. We start with the next two propositions on the $q$-numerical radii of operator matrices of rank at most one.

\begin{proposition}\label{related_prop_1}
	Let $a,b\in\H$ and $q\in[0,1]$. Then we have
    \[
        \wq\left(\begin{bmatrix}
		a\otimes b & 0 \\
		0 & 0
	\end{bmatrix}\right) = \wq(a\otimes b).
    \]
\end{proposition}
\begin{proof}
	For any $\begin{bmatrix}
		x\\
		y 
	\end{bmatrix}\in\H\oplus\H$, we have
    \[
    \begin{bmatrix}
		a\otimes b & 0 \\
		0 & 0
	\end{bmatrix}\begin{bmatrix}
		x\\ y 
	\end{bmatrix}=\begin{bmatrix}
		(a\otimes b)x\\ 0 
	\end{bmatrix}=\begin{bmatrix}
		\scal{x}{a}b\\ 0 
	\end{bmatrix}=\left\langle\begin{bmatrix}
			x  \\
			y
	\end{bmatrix}, \begin{bmatrix}
			a  \\
			0 
	\end{bmatrix}\right\rangle \begin{bmatrix}
		b  \\
		0 
	\end{bmatrix}=\left(\begin{bmatrix}
		a  \\
		0 
	\end{bmatrix}\otimes\begin{bmatrix}
		b  \\
		0 
	\end{bmatrix}\right)\begin{bmatrix}
		x  \\
		y 
	\end{bmatrix}.
    \]
	Therefore, the operator $$\begin{bmatrix}
		a\otimes b & 0 \\
		0 & 0
	\end{bmatrix}\in\mathcalb{B}(\H\oplus\H)$$ is of rank at most one. Theorem~\ref{osnovna} now implies
    \begin{align*}
        \pushQED{\qed}
        \wq\left(\begin{bmatrix}
				a\otimes b & 0 \\
				0 & 0
	\end{bmatrix}\right)&=\wq\left(\begin{bmatrix}
				a  \\
				0 
			\end{bmatrix}\otimes\begin{bmatrix}
				b  \\
				0 
			\end{bmatrix}\right)\\
			&=\frac{1}{2}\left(\left\|\begin{bmatrix}
				a  \\
				0 
			\end{bmatrix}\right\| \, \left\|\begin{bmatrix}
				b  \\
				0 
			\end{bmatrix}\right\|+q\left|\scal{\begin{bmatrix}
					a  \\
					0 
			\end{bmatrix}}{\begin{bmatrix}
					b  \\
					0 
			\end{bmatrix}}\right|\right)\\
            &\quad + \frac{\sqrt{1-q^2}}{2} \sqrt{\left\|\begin{bmatrix}
					a  \\
					0 
				\end{bmatrix}\right\|^2\left\|\begin{bmatrix}
					b  \\
					0 
				\end{bmatrix}\right\|^2-\left|\left\langle \begin{bmatrix}
					a  \\
					0 
				\end{bmatrix},\begin{bmatrix}
					b \\
					0 
				\end{bmatrix}\right\rangle\right|^2}\\
			&=\frac{\|a\|\|b\|+q|\langle a,b\rangle|}{2}+\frac{\sqrt{1-q^2}}{2} \sqrt{\|a\|^2\|b\|^2-|\langle a,b\rangle|^2}\\
            &=\wq(a\otimes b) \qedhere.
    \end{align*}
\end{proof}
\begin{remark}
    Proposition \ref{related_prop_1} does not hold for arbitrary operators, as noted in \cite[Corollary~5.8]{q1}.
\end{remark}

\begin{proposition}\label{related_prop_2}
	Let $a, b \in \mathcal{H}$ and $q \in [0,1]$. Then we have
	\[
        \wq\left(\begin{bmatrix}
		0 & a\otimes b \\
		0 & 0
	\end{bmatrix}\right)=\frac{1+\sqrt{1-q^2}}{2}\|a\| \|b\|.
    \]
\end{proposition}
\begin{proof}
	For any $\begin{bmatrix}
		x  \\
		y 
	\end{bmatrix}\in\H\oplus\H$, we have $$\begin{bmatrix}
		0 & a\otimes b \\
		0 & 0
	\end{bmatrix}\begin{bmatrix}
		x  \\
		y 
	\end{bmatrix}=\begin{bmatrix}
		(a\otimes b)y  \\
		0 
	\end{bmatrix}=\begin{bmatrix}
		\scal{y}{a}b  \\
		0 
	\end{bmatrix}=\scal{\begin{bmatrix}
			x  \\
			y
	\end{bmatrix}}{\begin{bmatrix}
			0  \\
			a 
	\end{bmatrix}}\begin{bmatrix}
		b  \\
		0 
	\end{bmatrix}=\left(\begin{bmatrix}
		0  \\
		a 
	\end{bmatrix}\otimes\begin{bmatrix}
		b  \\
		0 
	\end{bmatrix}\right)\begin{bmatrix}
		x  \\
		y 
	\end{bmatrix},$$
	Therefore, the operator $$\begin{bmatrix}
		0 & a\otimes b \\
		0 & 0
	\end{bmatrix}\in\mathcalb{B}(\H\oplus\H)$$is of rank at most one. Using Theorem \ref{osnovna}, we obtain
    \begin{align*}
    \pushQED{\qed}
    \wq\left(\begin{bmatrix}
            0 & a\otimes b \\
            0 & 0
        \end{bmatrix}\right)&=\wq\left(\begin{bmatrix}
            0  \\
            a 
        \end{bmatrix}\otimes\begin{bmatrix}
            b  \\
            0 
        \end{bmatrix}\right)\\
        &=\frac{1}{2}\left(\left\|\begin{bmatrix}
            0  \\
            a 
        \end{bmatrix}\right\| \, \left\|\begin{bmatrix}
            b  \\
            0 
        \end{bmatrix}\right\|+q\left|\scal{\begin{bmatrix}
                0  \\
                a 
        \end{bmatrix}}{\begin{bmatrix}
                b  \\
                0 
        \end{bmatrix}}\right|\right)\\
        &\qquad + \frac{\sqrt{1-q^2}}{2} \sqrt{\left\|\begin{bmatrix}
                0  \\
                a
                \end{bmatrix}\right\|^2\left\|\begin{bmatrix}
                b  \\
                0 
            \end{bmatrix}\right\|^2-\left|\left\langle \begin{bmatrix}
                0  \\
                a 
            \end{bmatrix},\begin{bmatrix}
                b \\
                0 
            \end{bmatrix}\right\rangle\right|^2}\\
        &=\frac{\|a\|\|b\|+ 0q}{2}+\frac{\sqrt{1-q^2}}{2} \sqrt{\|a\|^2\|b\|^2-0^2}\\
        &=\frac{1+\sqrt{1-q^2}}{2}\|a\|\|b\|. \qedhere
    \end{align*}
\end{proof}
\begin{remark}
    Proposition \ref{related_prop_2} can also be proved by using Corollary \ref{lambda_cor}.
\end{remark}

We proceed with the following result on the functional properties of $q \mapsto \omega_q(A)$ for a given rank-one operator $A\in\BH$.

\begin{theorem}
	Let $a, b \in \H\setminus\{0\}$ and let the function $f:[0,1]\to\R$ is defined by $$f(q) = \omega_q(a\otimes b)$$ for $q \in [0,1]$. Then we have the following:
	\begin{enumerate}[label=\textbf{(\roman*)}]
		\item $f$ is concave. In particular, if $a$ and $b$ are linearly independent, then $f$ is strictly concave.
        
		\item $f$ attains its maximum value uniquely at $\dfrac{\abs{\scal{a}{b}}}{\norm{a}\norm{b}}$, with $f\left( \dfrac{\abs{\scal{a}{b}}}{\norm{a}\norm{b}} \right) =  \|a\| \, \|b\|$.

		\item $f$ attains its minimum value at either $0$ or $1$ (or both).
		\item $f([0,1])=\left[\dfrac{\norm{a}\norm{b}+\min\left\{\abs{\scal{a}{b}}, \sqrt{\norm{a}^2\norm{b}^2-\abs{\scal{a}{b}}^2} \right\}}{2}, \norm{a}\norm{b}\right]$ .
	\end{enumerate}
\end{theorem}

\begin{proof}\phantom{}
	\begin{enumerate}[label=\textbf{(\roman*)}]
		\item By direct computation, Theorem \ref{osnovna} implies
		\begin{equation}\label{aux_4}
			f'(q)=\dfrac{|\scal{a}{b}|}{2}-\dfrac{q}{2\sqrt{1-q^2}}\sqrt{\norm{a}^2\norm{b}^2-\abs{\scal{a}{b}}^2}, \quad q\in(0,1),
		\end{equation}
		and
		\begin{equation}\label{aux_5}
			f''(q)=-\dfrac{1}{2(1-q^2)\sqrt{1-q^2}}\sqrt{\norm{a}^2\norm{b}^2-\abs{\scal{a}{b}}^2}\leqslant 0, \quad q\in(0,1).
		\end{equation}
		Therefore, $f$ is concave on $[0, 1]$. Moreover, if $a$ and $b$ are linearly independent, then the inequality in \eqref{aux_5} is strict, hence $f$ is strictly concave.

		\item Let $q_0=\dfrac{\abs{\scal{a}{b}}}{\norm{a}\norm{b}}$. From \eqref{aux_4}, a routine computation shows that $f'(q) > 0$ for any $q \in (0, q_0)$, and $f'(q) < 0$ for any $q \in (q_0, 1)$. The function $f$ thus attains its maximum value uniquely at $q_0$. Moreover, we have
		\begin{align*}
			f\left(q_0\right)
			&=\frac{\|a\|\|b\|+\frac{\abs{\scal{a}{b}}}{\norm{a}\norm{b}}|\langle a,b\rangle|}{2}+\frac{\sqrt{1-\left(\frac{\abs{\scal{a}{b}}}{\norm{a}\norm{b}}\right)^2}}{2} \sqrt{\|a\|^2\|b\|^2-|\langle a,b\rangle|^2}\\
			&=\dfrac{\norm{a}^2\norm{b}^2+\abs{\scal{a}{b}}^2}{2\norm{a}\norm{b}}+\dfrac{\norm{a}^2\norm{b}^2-\abs{\scal{a}{b}}^2}{2\norm{a}\norm{b}} =\norm{a}\norm{b}.
		\end{align*}

		\item This follows from the continuity of $f$ and part \textbf{(i)}.
        \pushQED{\qed}
		\item This follows from parts \textbf{(ii)} and \textbf{(iii)}.\qedhere
	\end{enumerate}
\end{proof}

For a Hilbert space $\H$ with two inner products $\scal{\cdot}{\cdot}_1$ and $\scal{\cdot}{\cdot}_2$, we say that $\scal{\cdot}{\cdot}_1 \preccurlyeq \scal{\cdot}{\cdot}_2$ if $\norm{x}_1\leqslant\norm{x}_2$ for all $x\in\H$. The following theorem investigates the monotonicity property of the $q$-numerical radii with respect to different inner products.

\begin{theorem}
Let $\H$ be a Hilbert space with the inner products $\scal{\cdot}{\cdot}_1$ and $\scal{\cdot}{\cdot}_2$. Also, let $a, b\in\H$ and $q\in[0,1]$, and let $\omega_q^1(\cdot)$ and $\omega_q^2(\cdot)$ denote the $q$-numerical radii corresponding to $\scal{\cdot}{\cdot}_1$ and $\scal{\cdot}{\cdot}_2$, respectively. If $$\scal{\cdot}{\cdot}_1 \preccurlyeq \scal{\cdot}{\cdot}_2\,\,\,\text{and}\,\,\,|\scal{a}{b}_1| \leqslant |\scal{a}{b}_2|,$$ then
\begin{equation}\label{aux_6}
	\omega_q^1(a\otimes b)\leqslant \omega_q^2(a\otimes b).
\end{equation}
\end{theorem}
\begin{proof}
The result trivially holds if $a = 0$ or $b = 0$. 

Now, assume that $a \neq 0$ and $b \neq 0$. It is easy to see that
\begin{align}\label{logicno}
\begin{split}
    \norm{a}_1^2\norm{b}_1^2-\abs{\scal{a}{b}_1}^2 &= \norm{b}_1^2\left(\norm{a}_1^2-\frac{\abs{\scal{a}{b}_1}^2}{\norm{b}_1^2}\right)\\
    &= \norm{b}_1^2\left\|a-\frac{\scal{a}{b}_1}{\norm{b}_1^2} \, b\right\|_1^2\\
    &=\norm{b}_1^2\inf\limits_{\lambda\in\mathbb{C}}\norm{a-\lambda b}_{1}^2.
\end{split}
\end{align}
Then, from \eqref{glavnaformula} and \eqref{logicno}, we have
\begin{equation}\label{transformacija}
    \wq^1(a\otimes b)=\frac{\|a\|_1\|b\|_1+q|\langle a,b\rangle_1|}{2}+\frac{\sqrt{1-q^2}}{2}\norm{b}_1\inf\limits_{\lambda\in\mathbb{C}}\norm{a-\lambda b}_{1}.
\end{equation}
Now, using a similar procedure as in \eqref{logicno}, we obtain
\begin{equation}\label{transformacija1}
    \wq^2(a\otimes b)=\frac{\|a\|_2\|b\|_2+q|\langle a,b\rangle_2|}{2}+\frac{\sqrt{1-q^2}}{2}\norm{b}_2\inf\limits_{\lambda\in\mathbb{C}}\norm{a-\lambda b}_{2}.
\end{equation}
Finally, from \eqref{transformacija} and \eqref{transformacija1}, we get \eqref{aux_6}.
\end{proof}

We end the paper by analyzing the \(q\)-numerical radius of functions of operators whose rank is at most one. More precisely, we present the corresponding upper bound for \(\omega_q(f(a\otimes b))\), where $f$ is an analytic function. In the following theorem, we take $0^0$ to be equal to $1$.

\begin{theorem}\label{last_th}
	Let $a,b\in\H$ and $q\in[0,1]$. Also, let $R> \norm{a} \norm{b}$, and let $f \colon D(0, R)\to\C$ be a function defined by
    \[
        f(\lambda) = \sum_{k=0}^{\infty}\alpha_k\lambda^k \quad \text{for all $\lambda\in D(0,R)$},
    \]
    where $D(0,R)$ is the open disk of radius $R$ centered at $0$ in the complex plane. Then 
	\begin{equation}\label{eq:function_q}
	    \wq (f(a\otimes b))\leqslant \wq (a \otimes b)\left|\sum\limits_{k=1}^{\infty} \alpha_k\scal{b}{a}^{k-1}\right| + q \, |\alpha_0|.
	\end{equation}
\end{theorem}
\begin{proof}
	For every $n\in\N$, let $p_n$ be the polynomial
	\begin{equation*}
        p_n(\lambda) = \sum\limits_{k=0}^{n}\alpha_k\lambda^k .
	\end{equation*}
    Observe that the operator $p_n(a\otimes b)$ can be represented in the form
	\[
        p_n(a\otimes b)=\left(\sum\limits_{k=1}^{n}\alpha_k\scal{b}{a}^{k-1}\right)a\otimes b + \alpha_0I.
    \]
	Since the sequence $(p_n(a\otimes b))_{n\in\N}$ converges to $f(a\otimes b)$ in the norm topology of the space $\BH$,  it follows that
    \begin{equation}\label{eq:limit_f}
      f(a\otimes b)=\left(\sum\limits_{k=1}^{\infty}\alpha_k\scal{b}{a}^{k-1}\right)a\otimes b + \alpha_0I.  
    \end{equation}
    Also, recall that for any $T\in\BH$ and any $q\in\overline{\D}$, the equality 
	\begin{equation}\label{s_equality}
	    \mathcal{W}_q(\alpha T+ \beta I)=\alpha\mathcal{W}_q(T)+ \beta q
	\end{equation}
	holds; see \cite[Proposition 3.1]{GauWu}.
    From \eqref{eq:limit_f} and \eqref{s_equality}, it follows that
	\begin{equation*}
		\begin{split}
			\mathcal{W}_q(f(a\otimes b))&=\left(\sum\limits_{k=1}^{\infty}\alpha_k\scal{b}{a}^{k-1}\right)\mathcal{W}_q \left(a\otimes b\right)+ \alpha_0 q .
		\end{split}
	\end{equation*}
	We obtain \eqref{eq:function_q} directly from here.
\end{proof}

\section*{Declarations}

\noindent{\bf{Funding}}\\
\noindent This research was supported by the Ministry of Science, Technological Development and Innovation of the Republic of Serbia, through the grant numbers 451-03-66/2024-03/200104, 451-03-137/2025-03/200102 and 451-03-137/2025-03/200116. I.\ Damnja\-nović was also supported by the Science Fund of the Republic of Serbia, grant \#6767, Lazy walk counts and spectral radius of threshold graphs --- LZWK.
		
\vspace{0.5cm}
		
\noindent{\bf{Availability of data and materials}}\\
\noindent No data was used to support this study.

\vspace{0.5cm}

\noindent{\bf{Competing interests}}\\
\noindent The authors declare that they have no competing interests.

\vspace{0.5cm}

\noindent{\bf{Author contribution}}\\
\noindent The work presented here was carried out in collaboration between all the authors. All authors contributed equally and significantly in writing this article. All authors have contributed to the manuscript. All authors have read and agreed to the published version of the manuscript.

\end{document}